\documentclass[notitlepage]{scrartcl}
\usepackage{enumitem}
\usepackage{amsmath}
\usepackage{amssymb}
\usepackage{amsthm}
\usepackage{abstract}
\usepackage{xcolor}
\usepackage{comment}
\usepackage{hyperref}
\usepackage{mathtools}
\hypersetup{
    colorlinks=true,
    linkcolor=blue,
    filecolor=magenta,      
    urlcolor=cyan
    }
\makeatletter
\newcommand*{\barfix}[2][.175ex]{%
  \mathpalette{\@barfix{#1}}{#2}%
}
\newcommand*{\@barfix}[3]{%
  \vbox{%
    \kern#1\relax
    \hbox{$#2#3\m@th$}%
  }%
}
\makeatother

\newtheorem{theorem}{Theorem}
\newtheorem{thm}{Theorem}[section]

\newtheorem{lemma}[thm]{Lemma}

\newtheorem{question}[thm]{Question}
\newcommand{\footremember}[2]{%
    \footnote{#2}
    \newcounter{#1}
    \setcounter{#1}{\value{footnote}}%
}
\newcommand{\footrecall}[1]{%
    \footnotemark[\value{#1}]%
} 

\usepackage[margin=1in]{geometry}
\title{\vspace{-2em}Percolation on high-dimensional product graphs}
\author{%
Sahar Diskin \footremember{alley}{School of Mathematical Sciences, Tel Aviv University, Tel Aviv 6997801, Israel. Emails: sahardiskin@mail.tau.ac.il, krivelev@tauex.tau.ac.il.}%
\and Joshua Erde \footremember{trailer}{Institute of Discrete Mathematics, Graz University of Technology, Steyergasse 30, 8010 Graz, Austria. Emails: erde@math.tugraz.at, kang@math.tugraz.at.}%
\and Mihyun Kang \footrecall{trailer}%
\and Michael Krivelevich \footrecall{alley}%
}

\begin{document}
\maketitle
\vspace{-2em}
\begin{abstract}
We consider percolation on high-dimensional product graphs, where the base graphs are regular and of bounded order. In the subcritical regime, we show that typically the largest component is of order logarithmic in the number of vertices. In the supercritical regime, our main result recovers the sharp asymptotic of the order of the largest component, and shows that all the other components are typically of order logarithmic in the number of vertices. In particular, we show that this phase transition is \textit{quantitatively} similar to the one of the binomial random graph.

This generalises the results of Ajtai, Koml\'os, and Szemer\'edi \cite{AKS81} and of Bollob\'as, Kohayakawa, and \L{u}czak \cite{BKL92} who showed that the $d$-dimensional hypercube, which is the $d$-fold Cartesian product of an edge, undergoes a phase transition quantitatively similar to the one of the binomial random graph.
\end{abstract}

\section{Introduction}
\subsection{Background and motivation}
In 1960, Erd\H{o}s and R\'enyi \cite{ER60} discovered the following fundamental phenomenon: the component structure of the binomial random graph $G(d+1,p)$\footnote{As we mainly consider $d$-regular graphs, we use the somewhat unusual notation of $G(d+1,p)$ instead of $G(n,p)$, to make the comparison of the results simpler.} undergoes a remarkable \emph{phase transition} around the probability $p=\frac{1}{d}$. More precisely, if we let $y=y(\epsilon)$ be the unique solution in $(0,1)$ of the equation
\begin{align}\label{survival prob}
    y=1-\exp\left(-(1+\epsilon)y\right),
\end{align}
then Erd\H{o}s and R\'enyi's work \cite{ER60} implies\footnote{In fact, Erd\H{o}s and R\'enyi worked in the closely related \emph{uniform} random graph model $G(d+1,m)$.} the following. Let $\epsilon>0$ be a small enough constant. If $p=\frac{1-\epsilon}{d}$, then with probability tending to one as $d$ tends to infinity, all the components of $G(d+1,p)$ are of order $O_{\epsilon}(\log d)$. On the other hand, if $p=\frac{1+\epsilon}{d}$, then with probability tending to one as $d$ tends to infinity, $G(d+1,p)$ contains a unique giant component of order $(1+o(1))yd$, where $y$ is defined according to \eqref{survival prob}, and all the other components of $G(d+1,p)$ are of order $O_{\epsilon}(\log d)$. We note that $y$ is the survival probability of a Galton-Watson tree with offspring distribution $Bin\left(d, \frac{1+\epsilon}{d}\right)$, and we have that $y=2\epsilon-O(\epsilon^2)$. The regime where $p=\frac{1-\epsilon}{d}$ is often referred to as the \textit{subcritical regime}, while if $p=\frac{1+\epsilon}{d}$ it is often called the \textit{supercritical regime}. We refer the reader to \cite{B01, FK16, JLr00} for a systematic coverage of random graphs.

We can think of the binomial random graph as perhaps the simplest example of a \emph{percolation} model. Percolation is a mathematical process, initially studied by Broadbent and Hammersley \cite{BH57} to model the flow of a fluid through a porous medium whose channels may be randomly blocked. The underlying mathematical model is simple: given a fixed \emph{host} graph $G$ and some probability $p \in (0,1)$, we consider the random subgraph $G_p$ of $G$ obtained by retaining every edge independently with probability $p$. The component structure of the percolated subgraph $G_p$ is of particular interest, and in this broader setting the phase transition described in $G(d+1,p)$ can be viewed as an example of a \emph{percolation threshold}. See \cite{BR06, G99, K82} for a comprehensive introduction to percolation theory. 

Percolation is often studied on lattice-like graphs, where in contrast to $G(d+1,p)$ there is some non-trivial underlying geometry controlling the potential adjacencies in the random subgraph. One particular percolation model that has received considerable interest is that of percolation on the $d$\textit{-dimensional hypercube} $Q^d$ --- the graph with the vertex set $V(Q^d)=\{0,1\}^d$, where two vertices are adjacent if they differ in exactly one coordinate.
It was conjectured by Erd\H{o}s and Spencer \cite{ES79} that $Q^d_p$ undergoes a similar phase transition to $G(d+1,p)$ when $p$ is around $\frac{1}{d}$. This conjecture has been confirmed by Ajtai, Koml\'os, and Szemer\'edi \cite{AKS81}, with subsequent work by Bollob\'as, Kohayakawa, and \L{}uczak \cite{BKL92}. In fact, \cite{AKS81, BKL92} showed that the two models undergo \textit{quantitatively} similar phase transition --- if $p=\frac{1-\epsilon}{d}$, for $\epsilon>0$ a small constant, then with probability tending to one as $d$ tends to infinity all the components in $Q^d_p$ are of logarithmic order; and, if $p=\frac{1+\epsilon}{d}$ then with probability tending to one as $d$ tends to infinity there exists a unique giant component in $Q^d_p$ whose order is $(1+o(1))y|V(Q^d)|$, where $y$ is defined according to \eqref{survival prob}, and all other components of $Q^d_p$ are of logarithmic order. A similar phenomenon has been shown in other random graph models, such as graphs with a fixed degree sequence \cite{MR95} and percolation on pseudo-random graphs \cite{FKM04}. We will informally refer to this as the \textit{Erd\H{o}s-R\'enyi component phenomenon}.

It is thus a natural question to ask whether such a phenomenon holds for percolation in a wider family of graphs. The $d$-dimensional hypercube embeds naturally into $d$-dimensional space, but we can also view $Q^d$ as a high-dimensional object in terms of its product structure, since $Q^d$ can be obtained as the \textit{Cartesian product} of $d$ copies of a single edge. 

Given an integer $t>0$ and a sequence of graphs $\left(G^{(i)}\right)_{i=1}^t$, the Cartesian product of $G^{(1)},\ldots, G^{(t)}$, denoted by $G=G^{(1)}\square \cdots \square G^{(t)}$ or $G=\square_{i=1}^{t}G^{(i)}$, is the graph with the vertex set
\begin{align*}
    V(G)=\left\{v=(v_1,v_2,\ldots,v_t) \colon v_i\in V(G^{(i)}) \text{ for all } i \in [t]\right\},
\end{align*}
and the edge set
\begin{align*}
   E(G)=\left\{uv \colon \begin{array}{l} \text{there is some } i\in [t] \text{ such that }  u_j=v_j\\
    \text{for all } j \neq i \text{ and } u_iv_i\in E(G^{(i)}) \end{array}\right\}.
\end{align*}
We call $G^{(1)},\ldots, G^{(t)}$ the \textit{base graphs of} $G$, and say that they are \textit{non-trivial} if $|V(G^{(i)})|>1$. 

It is perhaps not unreasonable to expect that percolation on other \emph{high-dimensional graphs} might display similar behaviour. Indeed, recently Lichev \cite{L22} considered percolation on high-dimensional product graphs, under the assumption that the \emph{isoperimetric constants} of the base graphs were not shrinking too quickly. The \textit{isoperimetric constant} $i(H)$, also known as the Cheeger constant, of a graph $H$ is given by $i(H)=\min_{\begin{subarray}{c}S\subseteq V(H),\\
|S|\le |V(H)|/2\end{subarray}}\frac{e(S, S^C)}{|S|}$. Lichev \cite{L22} showed that the component structure of random subgraphs of such graphs undergoes phase transition when $p$ is around $\frac{1}{d}$ where $d\coloneqq d(G)$ is the average degree of the host graph $G$ --- if $p=\frac{1-\epsilon}{d}$, then with probability tending to one as $d$ tends to infinity, the components of $G_p$ are of order $o(|V(G)|)$, and if $p=\frac{1+\epsilon}{d}$, then with probability tending to one as $d$ tends to infinity, there exists a giant component in $G_p$ taking a linear fraction of the vertices. Note that this result only gives a \textit{qualitative} description of the phase transition, in the sense that a largest component in the supercritical regime is shown to contain a linear fraction of the vertices, but neither its uniqueness nor the leading constant are determined, and a largest component in the subcritical regime is only shown to have sublinear order.

\subsection{Main results}
Under the assumption that the base graphs are regular and of bounded order, our main result recovers the sharp asymptotic order of the largest component and shows that it is typically unique, with all the other components typically being of logarithmic order. 
\begin{theorem}\label{supercritical regime}
Let $C>1$ be a constant and let $\epsilon>0$ be a sufficiently small constant. Let $\left(G^{(i)}\right)_{i=1}^t$ be a sequence of connected and regular graphs, such that for every $i\in [t]$, $1<\big|V\left(G^{(i)}\right)\big|\le C$ and the degree of $G^{(i)}$ is $d_i$. Let $G=\square_{i=1}^{t}G^{(i)}$ and let $p=\frac{1+\epsilon}{d}$, where $d\coloneqq d(G) =\sum_{i=1}^{t}d_i$ is the degree of $G$. Then, \textbf{whp}\footnote{With high probability, that is, with probability tending to one as $t$ tends to infinity --- the error terms are vanishing with growing $t$. For example, in $Q^d$, $t=d$ and the asymptotics are in $d$.},
there exists a unique giant component of order $\left(1+o(1)\right)y|V(G)|$ in $G_p$, where $y=y(\epsilon)$ is defined as in (\ref{survival prob}). Furthermore, \textbf{whp}, all the remaining components of $G_p$ are of order $O_{\epsilon}(\log |V(G)|)$.
\end{theorem}
Observe that the assumption that all the base graphs have bounded order guarantees that the asymptotic behaviour comes from the product structure itself and the number $t$ of graphs it multiplies. Furthermore, we note that in the setting of Theorem \ref{subcritical regime}, $d=\Theta_C(\log |V(G)|)$, and so the degree of the host graph $G$ is growing with its order. Let us also remark that the above theorem also captures the case of the $t$-dimensional hypercube --- there for all $i\in [t]$ we have that $G^{(i)}=K_2$ and $C=2$. 

Let us complete the picture by considering the subcritical regime.
\begin{theorem}\label{subcritical regime}
Let $G$ be a graph on $n$ vertices with maximum degree $d$, let $0<\epsilon<1$ be a constant and let $p=\frac{1-\epsilon}{d}$. Then, with probability tending to one as $n$ tends to infinity, all the components of $G_p$ are of order at most $\frac{9\log n}{\epsilon^2}$.
\end{theorem}
Let us stress that variants of Theorem \ref{subcritical regime} are by now well-known (see, for example, \cite[page 23]{H63} and \cite{NP10}). We include the proof of Theorem \ref{subcritical regime} for the sake of completeness, noting that it is relatively short, and utilises the Breadth First Search (BFS) algorithm.

\subsection{Comparison with previous results}
Theorems \ref{supercritical regime} and \ref{subcritical regime} together show that when the underlying asymptotic geometry of the graph arises from a product structure, where the base graphs are bounded and regular, the percolated graph exhibits the Erd\H{o}s-R\'enyi component phenomenon, that is, the model undergoes \textit{quantitatively} similar phase transition to that of $G(d+1,p)$ around $p=\frac{1}{d}$. This strengthens the result of Lichev \cite{L22}, and generalises the results for the hypercube \cite{AKS81, BKL92} while providing shorter and simpler proofs. 

Let us note that Lichev \cite{L22} does not assume regularity or bounded order of the base graphs. In a companion paper \cite{DEKK23} the authors show that if one allows the base graphs to be irregular, then there are product graphs for which in the subcritical regime the largest component is typically of polynomial order (and not logarithmic). Furthermore, if one allows the base graphs to be of unbounded order, then there are product graphs for which a largest component in the supercritical regime is typically of sublinear order. We elaborate more on the conditions assumed both here and in \cite{L22}, and their implications, in Section \ref{discussion}.

Furthermore, while the result of \cite{BKL92} relies on Harper's edge-isoperimetric inequality for the hypercube \cite{H64}, which in particular implies that sets of polylogarithmic order in $Q^d$ have edge-expansion by a factor of $(1-o(1))d$, our proof only uses a much weaker and simpler isoperimetric inequality (see Theorem \ref{productiso}), which guarantees only edge-expansion by a constant factor. Indeed, our main innovation here, which avoids the need for a strong isoperimetric inequality, is a careful implementation of a simple yet powerful projection lemma (see Section \ref{s:proj}), and of a multi-round sprinkling argument (see Lemma \ref{the gap}). We expect these methods to be useful in other contexts where the analysis is constrained by the lack of a strong enough isoperimetric inequality. In particular, the proofs in this paper should be relatively easy to modify to the setting of site percolation on high-dimensional product graphs. We discuss this further in Section \ref{discussion}.

\subsection{Structure of the paper}
In Section \ref{S:prelim}, we introduce some notation, terminology and preliminary lemmas which will serve us throughout the rest of the paper. In Sections \ref{S: supercritical} and \ref{S: subcritical} we prove Theorems \ref{supercritical regime} and \ref{subcritical regime}, respectively. Finally, in Section \ref{discussion} we discuss our results and avenues for future research.
\section{Preliminaries}\label{S:prelim}
\subsection{Notation and terminology}
Let us introduce some notation and terminology, which we will use throughout the rest of the paper.

Recall that given a product graph $G=\square_{i=1}^tG^{(i)}$, we call the $G^{(i)}$ the \textit{base graphs} of $G$. Given a vertex $u = (u_1,u_2, \ldots, u_t)$ in $V(G)$ and $i \in [t]$ we call the vertex $u_i\in V(G^{(i)})$ the \textit{$i$-th coordinate} of $G$. As is standard, we may still enumerate the vertices of a given set $M$, such as $M=\left\{v_1,\ldots, v_m\right\}$ with $v_i\in V(G)$. Whenever confusion may arise, we will clarify whether the subscript stands for enumeration of the vertices of the set, or for their coordinates.
When $G^{(i)}$ is a graph on a single vertex, that is, $G^{(i)}=\left(\{u\},\varnothing\right)$, we call it \textit{trivial} (and \textit{non-trivial}, otherwise). We define the \textit{dimension} of $G=\square_{i=1}^tG^{(i)}$ to be the number of base graphs $G^{(i)}$ of $G$ which are non-trivial. Note that in Theorem \ref{supercritical regime}, we assumed that the base graphs have more than one vertex, implying that they are non-trivial. Given $H\subseteq G=\square_{i=1}^tG^{(i)}$, we call $H$ a \textit{projection of} $G$ if $H$ can be written as $H=\square_{i=1}^tH^{(i)}$ where for every $1\le i\le t$, $H^{(i)}=G^{(i)}$ or $H^{(i)}=\{v_i\}\subseteq V(G^{(i)})$; that is, $H$ is a projection of $G$ if it is the Cartesian product graph of base graphs $G^{(i)}$ and their trivial subgraphs. In that case, we further say that $H$ is the projection of $G$ onto the coordinates corresponding to the trivial subgraphs. For example, let $u_i\in V(G^{(i)})$ for $1\le i\le k$, and let $H=\{u_1\}\square\cdots\square\{u_k\}\square G^{(k+1)}\square\cdots\square G^{(t)}$. In this case we say that $H$ is a projection of $G$ onto the first $k$ coordinates, and its dimension is $t-k$.

Throughout the rest of the paper, unless explicitly stated otherwise, we let $C>1$ be an integer, let $G^{(1)},\ldots, G^{(t)}$ be regular and connected graphs with $1<\big|V\left(G^{(i)}\right)\big|\le C$ for all $i\in[t]$, and let $G=\square_{i=1}^{t}G^{(i)}$. We write $t$ for the dimension of $G$, $d\coloneqq \sum_{i=1}^{t}d\left(G^{(i)}\right)$ for the degree of $G$ and $n\coloneqq  |V(G)|$. 

We assume throughout the paper that $t\to\infty$, and our asymptotic notation will be with respect to $t$. Given a graph $H$ and a vertex $v \in V(H)$, we denote by $C_v(H)$ the component of $v$ in $H$, and by $N_H(v)$ the neighbours of $v$ in $H$. All logarithms are with the natural base. We omit rounding signs for the sake of clarity of presentation.

\subsection{The BFS algorithm} \label{BFS description}
For the proofs of our main results, we will use the Breadth First Search (BFS) algorithm. This algorithm explores the components of a graph $G$ by building a maximal spanning forest. 

The algorithm maintains three sets of vertices: 
\begin{itemize}
\item $S$, the set of vertices whose exploration is complete; 
\item $Q$, the set of vertices currently being explored, kept in a queue; and
\item $T$, the set of vertices that have not been explored yet.
\end{itemize}
The algorithm receives as input a graph $G$ and a linear ordering $\sigma$ on its vertices. It starts with $S=Q=\emptyset$ and $T=V(G)$, and ends when $Q\cup T=\emptyset$. At each step, if $Q$ is non-empty, the algorithm queries the vertices in $T$, in the order $\sigma$, to ascertain if they are neighbours in $G$ of the first vertex $v$ in $Q$. Each neighbour discovered is added to the back of the queue $Q$. Once all neighbours of $v$ have been discovered, we move $v$ from $Q$ to $S$. If $Q=\emptyset$, we move the next vertex from $T$ (according to $\sigma$) into $Q$. Note that the set of edges queried during the algorithm forms a maximal spanning forest of $G$.

In order to analyse the BFS algorithm on a random subgraph $G_p$ of a graph $G$ with $m$ edges, we will utilise the \emph{principle of deferred decisions}. That is, we will take a sequence $(X_i \colon 1 \leq i \leq m)$ of i.i.d Bernoulli$(p)$ random variables, which we will think of as representing a positive or negative answer to a query in the algorithm. When the $i$-th edge of $G$ is queried during the BFS algorithm we will include it in $G_p$ if and only if $X_i=1$.

\subsection{Preliminary Lemmas}
We will make use of two standard probabilistic bounds. The first one is a typical Chernoff type tail bound on the binomial distribution (see, for example, Appendix A in \cite{AS16}).
\begin{lemma}\label{chernoff}
Let $n\in \mathbb{N}$, let $p\in [0,1]$, and let $X\sim Bin(n,p)$. Then for any $0<t\le \frac{np}{2}$, 
\begin{align*}
    &\mathbb{P}\left[|X-np|\ge t\right]\le 2\exp\left(-\frac{t^2}{3np}\right).
\end{align*}
\end{lemma}

The second one is the well-known Azuma-Hoeffding inequality (see, for example, Chapter 7 in \cite{AS16}),
\begin{lemma}\label{azuma}
    Let $X = (X_1,X_2,\ldots, X_m)$ be a random vector with range $\Lambda = \prod_{i \in [m]} \Lambda_i$, and let $f:\Lambda\to\mathbb{R}$ be such that there exists $C \in \mathbb{R}^m$ such that for every $x,x' \in \Lambda$ which differ only in the $i$-th coordinate,
    \begin{align*}
        |f(x)-f(x')|\le C_i.
    \end{align*}
    Then, for every $t\ge 0$,
    \begin{align*}
        \mathbb{P}\left[\big|f(X)-\mathbb{E}\left[f(X)\right]\big|\ge t\right]\le 2\exp\left(-\frac{t^2}{2\sum_{i=1}^nC_i^2}\right).
    \end{align*}
\end{lemma}

The following result of Chung and Tetali \cite[Theorem 2]{CT98}, which gives a fairly simple lower bound on the isoperimetric constant of product graphs, will suffice for our proof.

\begin{thm}[\cite{CT98}]\label{productiso}
Let $G^{(1)},\ldots, G^{(t)}$ be graphs and let $G = \square_{j=1}^{t}G^{(j)}$. Then $i(G) \geq \frac{1}{2}\min_j \left\{i\left( G^{(j)}\right) \right\}$.
\end{thm}

Finally, we will use the following bound on the number of $k$-vertex trees in a bounded-degree graph.
\begin{lemma}[{\cite[Lemma 2]{BFM98}}]\label{BFM98}
Let $G$ be a graph on $n$ vertices with maximum degree $d$. Let $t_k(G)$ be the number of $k$-vertex trees in $G$. Then,
\begin{align*}
  t_k(G)\le n(ed)^{k-1}.
\end{align*}
\end{lemma}

\section{Supercritical regime}\label{S: supercritical}
Let us start by giving a brief sketch of the proof in this section.

We start in Section \ref{s:proj} by giving a useful technical lemma, which we call a \emph{projection lemma} (Lemma \ref{seperation lemma}), which allows us to cover a small set of points $M$ in the product graph with a set of pairwise disjoint projections of large dimension, each of which contains a unique point in $M$. This allows us to explore the graph `locally' around each of these points in an independent manner. 

Using this, in Section \ref{medium}, we show that \textbf{whp} a fixed proportion of the vertices in $G_p$ will be contained in \emph{big} components, of order at least $d^{16}$ (where the constant $16$ has been chosen rather arbitrarily). Then, in Section \ref{gap}, we show that these big components are typically `dense' in the host graph, in the sense that every vertex in $G$ is close to some big component of $G_p$. Finally, in Section \ref{proof}, we use a sprinkling argument to show that these big components merge into a unique giant component and determine its asymptotic order. This approach is relatively standard, and is roughly the method used by Ajtai, Koml\'os, and Szemer\'edi \cite{AKS81} and Bollob\'as, Kohayakawa, and \L{}uczak \cite{BKL92}, but the arguments are simplified substantially by our projection lemma.

However, these methods inherently can only establish a superlinear polynomial bound (in $d$) on the order of the second-largest component. In order to overcome that, the standard approach requires a strong isoperimetric inequality to demonstrate a gap in the size of the components. Here instead we utilise the projection lemma in an inductive manner, together with a carefully chosen multi-round sprinkling, in order to keep track more precisely of the properties of most of the vertices in big components. In particular, we show that any large enough connected set in $G_p$, where we only require these sets to be linearly large in $d$, must be adjacent to many vertices in big components. A further sprinkling argument then allows us to give an improved bound on the size of the second-largest component, completing the proof of Theorem \ref{supercritical regime}, without the need for a strong isoperimetric inequality.

\subsection{Projection lemma}\label{s:proj}
We begin by establishing the following fairly simple projection lemma (see \cite[Claim 2.2]{DK22} for its version for the hypercube), which will be useful throughout all of the subsequent sections.
\begin{lemma}\label{seperation lemma}
Let $M\subseteq V(G)$ be such that $|M|=m\le t$. Then, there exist pairwise disjoint projections $H_1, \ldots, H_m$ of $G$, each having dimension at least $t-m+1$, such that every $v\in M$ is in exactly one of these projections.
\end{lemma}
\begin{proof}
Recall that $t$ is the dimension of the graph. We argue by induction on pairs $(m,t)$ with $m \leq t$ under the lexicographical ordering. When $m=1$, we simply take $H_1=G$.

Let $M \subseteq V(G)$ have size $m\geq 2$ and let us assume the statement holds for all $(m',t')<(m,t)$. Let us write $M=\left\{v_1,\ldots, v_m\right\}$, where we stress that 
the subscript here is an enumeration of the vertices of $M$ and not the coordinates of a fixed vertex. 

There is some coordinate $i$ on which at least two of the vertices of $M$ do not agree. Let us denote by $M_i$ the set of the $i$-th coordinates of the vertices of $M$, that is, $V(G^{(i)})\supseteq M_i=\left\{v_{1,i},\ldots, v_{\ell,i}\right\}$, where we may re-order the vertices so that $v_{j,i}$ is the $i$-th coordinate of the vertex $v_j\in M$, and $2\le \ell \le m$ (note that it is possible that $\ell<m$, since there could be vertices in $M$ which agree on their $i$-th coordinates).

Let us now consider the pairwise disjoint projections $H_1, \ldots, H_{\ell}$ of $G$ defined by
\begin{align*}
    H_{j}=G^{(1)}\square\cdots\square G^{(i-1)}\square \{v_{j,i}\}\square G^{(i+1)}\square \cdots \square G^{(t)}.
\end{align*}
That is, in the $j$-th projection, we take the $i$-th coordinate of $H_{j}$ to be the trivial graph $\{v_{j,i}\}$ (which is the $i$-th coordinate of the $j$-th vertex in $M_i$).

Note that each of these projections has dimension $t-1$ and contains at least one vertex of $M$, and that each of the vertices of $M$ is in exactly one of these projections. Hence, each such projection contains at most $m-1$ vertices from $M$. 

We can thus apply the induction hypothesis to each of these projections, giving rise to $m$ pairwise disjoint projections of $G$, each of dimension at least $(t-1)-(m-1)+1=t-m+1$ and containing 
exactly one vertex from $M$.
\end{proof}

\subsection{Vertices in large components}\label{medium}
We first estimate from below the probability that a vertex of $G$ lies in a polynomially sized (in $d$) component of $G_p$, with the proof inspired by \cite{BKL94}.
\begin{lemma}\label{medium comp}
Let $\epsilon>0$ be constants and let $p\geq\frac{1+\epsilon}{d}$. Let $r >0$ be an integer and let $m_r=d^{\frac{r}{4}}$. Then, there exists a constant $c=c(\epsilon,r)>0$ such that for any $v\in V(G)$,
\begin{align*}
    \mathbb{P}\left[\big|C_v(G_p)\big|\ge cm_r\right]\ge y-o_t\left(1\right),
\end{align*}
where $y=y(\epsilon)$ is as defined in (\ref{survival prob}).
\end{lemma}
\begin{proof}
We will prove the slightly more explicit statement that the result holds with ${c(\epsilon,r)= \left(\frac{y(\epsilon)}{5}\right)^{r}}$ by induction on $r$, over all possible values of $C$ and $\epsilon$, and all choices of $G^{(1)},\ldots, G^{(t)}$. 

For $r=1$, we run the BFS algorithm (as described in Section \ref{BFS description}) on $G_p$ starting from $v$ with a slight alteration: we terminate the algorithm once $\min\left(\big|C_v(G_p)\big|,d^{\frac{1}{2}}\right)$ vertices are in $S\cup Q$. Note that at every point in the algorithm we have $|S\cup Q|\le d^{\frac{1}{2}}$, and therefore at each point in the algorithm the first vertex $u$ in the queue has at least $d-d^{\frac{1}{2}}$ neighbours (in $G$) in $T$. Hence, we can couple the forest $F$ built by this truncated BFS process with a Galton-Watson tree $B$ rooted at $v$ with offspring distribution $Bin\left(d-d^{\frac{1}{2}},p\right)$ such that $B \subseteq F$ as long as $|B| \leq d^{\frac{1}{2}}$.

Since $\left(d-d^{\frac{1}{2}}\right)\cdot p \ge 1+\epsilon-o(1)$, standard results (see, for example, \cite[Theorem 4.3.12]{D19}) imply that $B$ grows infinitely large with probability $y-o(1)$. Thus, with probability at least $y-o(1)$, we have that 
\[
\big|C_v(G_p)\big|\ge d^{\frac{1}{2}} \geq \left(\frac{y}{5}\right)d^{\frac{1}{4}} = c(\epsilon,1)m_1.
\]

Let $r\geq 2$ and let us assume that the statement holds with $c(\epsilon',r-1) = \left(\frac{y(\epsilon')}{5}\right)^{r-1}$ for all $C,\epsilon$ and $G^{(1)},\ldots, G^{(t)}$. We will argue via a double sprinkling. Set $p_2=d^{-\frac{5}{4}}$ and $p_1=\frac{p-p_2}{1-p_2}$ so that $(1-p_1)(1-p_2)=1-p$. Note that $G_p$ has the same distribution as $G_{p_1}\cup G_{p_2}$, and that $p_1=\frac{1+\epsilon'}{d}$ where $\epsilon' = \epsilon - o(1)$. In fact, we will not expose either $G_{p_1}$ or $G_{p_2}$ all at once, but in several stages, each time considering only some subset of the edges.  

We begin in a manner similar to the case of $r=1$. We run the BFS algorithm on $G_{p_1}$ starting from $v$, and we terminate the exploration once $\min\left(\big|C_v(G_{p_1})\big|,d^{\frac{1}{2}}\right)$ vertices are in $S\cup Q$. Once again, by standard arguments, we have that $\big|C_v(G_{p_1})\big|\ge d^{\frac{1}{2}}$ with probability at least $y\left(\epsilon'\right)-o(1)=y-o(1)$. Let us write $W_0\subseteq C_v(G_{p_1})$ for the set of vertices explored in this process, and assume in what follows that $W_0$ is of order $d^{\frac{1}{2}}$. Using Lemma \ref{seperation lemma}, we can find pairwise disjoint projections $H_1,\ldots, H_{d^{\frac{1}{2}}}$ of $G$, each having dimension at least $t-d^{\frac{1}{2}}$, such that each $v\in W_0$ is in exactly one of the $H_i$. We denote the vertex of $W_0$ in $H_i$ by $v_i$.

Now, by our assumptions on $G$, we have that $|V\left(G^{(j)}\right)|\le C$ for every $j\in[t]$ and hence clearly $d\left(G^{(j)}\right)\le C$. Thus, each of the $H_i$ is $d_i$-regular with $d_i\ge d-Cd^{\frac{1}{2}} = (1-o(1))d$. In particular, each $v_i$ has $(1-o(1))d$ neighbours in $H_i$. Let us define the following set of vertices:
\begin{align*}
    W=\bigcup_{i\in [d^{\frac{1}{2}}]} N_{H_i}(v_i).
\end{align*}
Then, $W\subseteq N_G(W_0)$ and since the $H_i$ are pairwise disjoint we have $|W|\ge d^{\frac{1}{2}}(1-o(1))d\geq \frac{9d^{\frac{3}{2}}}{10}$. We now look at the edges in $G_{p_2}$ between $W_0$ and $W$. Let us denote the vertices in $W$ that are connected with $W_0$ in $G_{p_2}$ by $W'$. Since $|W|\ge \frac{9d^{\frac{3}{2}}}{10}$ and $p_2=d^{-\frac{5}{4}}$, $|W'|$ stochastically dominates $Bin\left(\frac{9d^{\frac{3}{2}}}{10}, d^{-\frac{5}{4}}\right)$. Thus, by Lemma \ref{chernoff}, we have that $|W'|\ge \frac{d^{\frac{1}{4}}}{3}$ and $|W'|\le d^{\frac{1}{2}}$ with probability at least $1-\exp\left(-\frac{d^{\frac{1}{4}}}{15}\right)=1-o(1)$. In what follows we will assume that these two inequalities hold.

Let $W'_i=W'\cap V(H_i)$. Now, for each $i$, we apply Lemma \ref{seperation lemma} to find a family of $\ell_i\le d^{\frac{1}{2}}$ pairwise disjoint projections of $H_i$, we denote them by $H_{i,1}, \ldots, H_{i,\ell_i}$, such that every vertex of $W'_i$ is in exactly one of the $H_{i,j}$, and each of the $H_{i,j}$ is of dimension at least $t-2d^{\frac{1}{2}}$. Finally, we denote by $v_{i,j}$ the unique vertex of $W'_i$ that is in $H_{i,j}$. Note that, each $H_{i,j}$ is $d_{i,j}$-regular with $d_{i,j} \geq d - 2Cd^{\frac{1}{2}}$.

Crucially, observe that when we ran the BFS algorithm on $G_{p_1}$, we did not query any of the edges in any of the $H_{i,j}$ - we only queried edges in $W_0$ and between $W_0$ and its neighbourhood, and by construction $E(H_{i,j})\cap E\left(W_0\cup N_G(W_0)\right)=\varnothing$. Note that, since $y=y(\epsilon)$, defined in (\ref{survival prob}), is a continuous increasing function of $\epsilon$ on $(0,\infty)$ and $p_1\cdot d_{i,j} =1 + \epsilon_{i,j} \geq 1 + \epsilon - o(1)$ for all $i,j$, we may apply the induction hypothesis to $v_{i,j}$ in $G_{p_1}\cap H_{i,j}$ and conclude that 
\[
|C_{v_{i,j}}\left(H_{i,j}\cap G_{p_1}\right)| \geq c(\epsilon_{i,j},r-1) d^{\frac{r-1}{4}} \geq (1-o(1))c(\epsilon,r-1) d^{\frac{r-1}{4}}
\]
with probability at least $y\left(\epsilon_{i,j}\right)-o(1)\geq y-o(1)$. Note that these events are independent for each $H_{i,j}$. Hence, since by assumption $|W'| \geq \frac{d^{\frac{1}{4}}}{3}$, it follows from Lemma \ref{chernoff} that \textbf{whp} at least $\frac{y d^{\frac{1}{4}}}{4}$ of these $v_{i,j}$ have that $|C_{v_{i,j}}\left(H_{i,j}\cap G_{p_1}\right)|\ge (1-o(1))c(\epsilon,r-1) d^{\frac{r-1}{4}}$.

As such, we may conclude that with probability at least $y-o(1)$, we have that 
\begin{align*}
    |C_v(G_p)|\ge (1-o(1))\frac{yd^{\frac{1}{4}}}{4}c(\epsilon,r-1) d^{\frac{r-1}{4}} \geq  \left( \frac{y}{5}\right)^{r} d^\frac{r}{4} = c(\epsilon,r)m_r,
\end{align*}
completing the induction step.
\end{proof}

We can now estimate the number of vertices in components of order at least $d^{16}$.
\begin{lemma}\label{big comp}
Let $\epsilon >0$ and let $p= \frac{1+\epsilon}{d}$. Let $W\subseteq V(G)$ be the set of vertices belonging to components of order at least $d^{16}$ in $G_p$. Then, \textbf{whp},
\begin{align*}
    |W|=\left(1+o(1)\right)yn,
\end{align*}
where $y=y(\epsilon)$ is as defined in (\ref{survival prob}).
\end{lemma}
\begin{proof}
By Lemma \ref{medium comp}, applied with $r=4\cdot 16+1$, every $v\in V(G)$ is contained in a component of order at least $d^{16}$ in $G_p$ with probability at least $y-o(1)$. Thus, $\mathbb{E}\left[|W|\right]\ge \left(1-o(1)\right)yn$.

Furthermore, since $G$ is $d$-regular, for every $v\in V(G)$ an easy coupling implies that $|C_v(G_p)|$ is stochastically dominated by the number of vertices in a Galton-Watson tree with offspring distribution $Bin(d,p)$. Thus, by standard arguments (see, for example, \cite[Theorem 4.3.12]{D19}), we have that for every $v\in V(G)$, $\mathbb{P}\left[|C_v(G_p)|\ge d^{16}\right]\le y+o_d(1)$.

Since $d$ tends to infinity with $t$, we obtain that $\mathbb{E}\left[|W|\right]\le \left(1+o(1)\right)yn$, and we conclude that $\mathbb{E}\left[|W|\right]=\left(1+o(1)\right)yn.$

It remains to show that $|W|$ is concentrated about its mean. To this end, let us consider the edge-exposure martingale on $G_p$ (see, for example, \cite{AS16}[Chapter 7]), whose length is $|E(G)| = \frac{nd}{2}$. Adding or deleting an edge can change $W$ by at most $2d^{16}$ vertices. Thus, a standard application of Lemma \ref{azuma} implies that
\begin{align*}
    \mathbb{P}\left[\big||W|- \mathbb{E}\left[|W|\right]\big|\ge n^{\frac{2}{3}}\right]&\le 2\exp\left(-\frac{n^{\frac{4}{3}}}{2\sum_{i=1}^{\frac{nd}{2}}4d^{32}}\right) \le2\exp\left(-\frac{n^{\frac{1}{3}}}{5d^{33}}\right)=o(1),
\end{align*}
since $d=\Theta_C(\log n)$.
\end{proof}

\subsection{Big components are everywhere-dense}\label{gap}
In this section, we establish several lemmas showing that, typically, big components in $G_p$ are \emph{everywhere-dense} in $G$. 

We begin with the first type of density lemma. Note that, in particular, the following implies that for any fixed vertex $v \in G$, \textbf{whp} $v$ is adjacent to many vertices which lie in big components in $G_p$. 
\begin{lemma} \label{dense1}
Let $\epsilon>0$ be a small enough constant, let $p=\frac{1+\epsilon}{d}$ and let $M\subseteq V(G)$ be such that $|M|=m\le \frac{\epsilon d}{10C}$. Then, the probability that every $v\in M$ has less than $\frac{\epsilon^2d}{40C}$ neighbours (in $G$) which lie in components of $G_p$ whose order is at least $d^{16}$ is at most $\exp\left(-\frac{\epsilon^2dm}{40C}\right)$.
\end{lemma}
\begin{proof}
Let $M=\{u_1,\ldots,u_m\}$, where we stress here that the subscript denotes an enumeration of the vertices of $M$, that is, $u_i\in V(G)$. By Lemma \ref{seperation lemma}, we can find pairwise disjoint projections $H_1, \ldots, H_m$ of $G$ such that each $H_i$ is of dimension at least $t-m+1\ge t-\frac{\epsilon d}{10C}$ and $u_i\in V(H_i)$ for all $i\in[m]$. Note that since $d\le Ct$, $t-\frac{\epsilon d}{10C}>0$.

Let us fix some $i$. Without loss of generality, we may assume that $H_i$ is the projection of $G$ onto the last $m_i\le m-1$ coordinates, that is,
\begin{align*}
    H_i=G^{(1)}\square\cdots\square G^{(t-m_i)}\square\{u_{i,t-m_i+1}\}\square\cdots\square\{u_{i,t}\},
\end{align*}
where $u_{i,\ell}\in V(G^{(\ell)})$ is the $\ell$-th coordinate of $u_i$, that is, the $\ell$-th coordinate of the $i$-th vertex of $M$. By our assumption, each $G^{(\ell)}$ has at least $2$ vertices and is connected. Thus, for each $\ell$, we can choose arbitrarily one of the neighbours of $u_{i,\ell}$ in $G^{(\ell)}$ and denote it by $v_{i,\ell}$, where the subscript in $v_{i,\ell}$ is to stress that it is a neighbour of $u_{i,\ell}$ in $G^{(\ell)}$. We now define the following $\frac{\epsilon d}{10C}$ pairwise disjoint projections of $H_i$, denoted by $H_i(1),\ldots , H_i\left(\frac{\epsilon d}{10C}\right)$, each having dimension at least $t-\frac{\epsilon d}{5C}$. We set $H_i(j)$ to be the projection of $H_i$ on the first $\frac{\epsilon d}{10C}$ coordinates, such that the $j$-th coordinate (where $1\le j\le \frac{\epsilon d}{10C}$) is the trivial subgraph $\{v_{i,j}\}\subseteq G^{(j)}$, and the coordinates $1\le \ell\le \frac{\epsilon d}{10C}$ (where $\ell\neq j$) are the trivial subgraphs $\{u_{i,\ell}\}\subseteq G^{(\ell)}$. Note that each $H_i(j)$ is at distance $1$ from $u_i$, since it contains the vertex $$v_i(j)\coloneqq \left(u_{i,1},\cdots,u_{i,j-1},v_{i,j},u_{i,j+1},\cdots,u_{i,t}\right).$$

Observe that $p=\frac{1+\epsilon}{d}$ is supercritical for every $H_i(j)$, since $d(H_i(j))\ge d-C\cdot \frac{\epsilon d}{5C}=\left(1-\frac{\epsilon}{5}\right)d$, and so $p \cdot d(H_i(j))\ge 1+\frac{3\epsilon}{5}$, for small enough $\epsilon$. Then, by Lemma \ref{medium comp}, with probability at least $y\left(\frac{3\epsilon}{5}\right)-o(1)$ we have that $v_i(j)$ belongs to a component of order at least $d^{16}$ in $G_p\cap H_i(j)$, and we note that by $(\ref{survival prob})$, we have that $y\left(\frac{3\epsilon}{5}\right)-o(1)>\epsilon$ for small enough $\epsilon$. These events are independent for different $j$, and thus by Lemma \ref{chernoff}, with probability at least $1-\exp\left(-\frac{\epsilon^2d}{40C}\right)$, at least $\frac{\epsilon^2d}{40C}$ of the $v_i(j)$ belong to a component of order at least $d^{16}$. 

These events are independent for different $i$, and thus the probability that none of the $u_i$ have at least $\frac{\epsilon^2d}{40C}$ neighbours (in $G$) in components of $G_p$ whose order is at least $d^{16}$ is at most $\exp\left(-\frac{\epsilon^2dm}{40C}\right)$, as required.
\end{proof}

Hence, we expect almost all vertices in $G$ to be adjacent to a vertex in a big component in $G_p$. However, even the vertices not adjacent to big components, whose proportion is likely to be small, will typically be not too far from a big component.
\begin{lemma}\label{dense3}
Let $\epsilon>0$ be a small enough constant and let $p=\frac{1+\epsilon}{d}$. Then, \textbf{whp}, every $v\in V(G)$ is at distance (in $G$) at most $2$ from a component of order at least $d^{16}$ in $G_p$.
\end{lemma}
\begin{proof}
Fix $v\in V(G)$. Let $M$ be an arbitrarily chosen set of $\frac{\epsilon d}{10C}$ neighbours of $v$ in $G$. Then, by Lemma \ref{dense1}, the probability that none of the vertices of $M$ have a neighbour (in $G$) in a component of order at least $d^{16}$ in $G_p$ is at most $\exp\left(-\frac{\epsilon^3d^2}{400C^2}\right)$. Thus, by a union bound over all possible choices of $v\in V(G)$, the probability that there is a vertex which is not at distance (in $G$) at most two from a component of order at least $d^{16}$ in $G_p$ is at most $n\exp\left(-\frac{\epsilon^3d^2}{400C^2}\right)\le \exp\left(d\log C-\frac{\epsilon^3d^2}{400C^2}\right)=o(1)$, where we used that $n\le C^d$.
\end{proof}

Finally, we can use Lemma \ref{dense1} and Lemma \ref{BFM98} to show the following.
\begin{lemma}\label{dense2}
Let $\epsilon>0$ be a small enough constant and let $p=\frac{1+\epsilon}{d}$. Let $W$ be the set of vertices in $G_p$ belonging to components of order at least $d^{16}$ and let $C_1>0$ be a constant. Then, \textbf{whp} every connected subset $M \subseteq V(G)$ of size $C_1d$ contains at least $C_1d-\frac{\epsilon d}{10C}$ vertices $v\in M$ with $|N_G(v)\cap W|\ge\frac{\epsilon^2d}{40C}$.
\end{lemma}
\begin{proof}
By Lemma \ref{BFM98}, there are at most $n(ed)^{C_1d}$ connected subsets $M \subseteq V(G)$, and there are at most $\binom{C_1d}{\frac{\epsilon d}{10C}} \leq 2^{C_1d}$ ways to choose a subset $X\subset M$ of size $\frac{\epsilon d}{10C}$. By Lemma \ref{dense1}, the probability that no vertex in $X$ has at least $\frac{\epsilon^2d}{40C}$ neighbours in $W$ is at most $\exp\left(-\frac{\epsilon^3d^2}{400C^2}\right)$.

Hence, by a union bound, the probability that the conclusion of the lemma does not hold is at most $n(ed)^{C_1d}2^{C_1d}\exp\left(-\frac{\epsilon^3d^2}{400C^2}\right) =o(1)$, where we used that $d = \Theta_C(\log n)$.
\end{proof}

\subsection{Proof of Theorem \ref{supercritical regime}}\label{proof}
Throughout this section, we take $C$, $G^{(1)},\ldots, G^{(t)}$, $G$, $\epsilon$, $d$ and $p$ to be as in the statement of Theorem \ref{supercritical regime}, and let $y=y(\epsilon)$ be as defined in (\ref{survival prob}). In order to prove Theorem \ref{supercritical regime}, we will use a multi-round sprinkling argument.

More precisely, let $p_2=\frac{\epsilon}{2d}$, $p_3=\frac{1}{d^2}$, and let $p_1$ be such that $(1-p_1)(1-p_2)(1-p_3)=1-p$. Note that $p_1=\frac{1+\frac{\epsilon}{2}+o(1)}{d}$. Let $G_{p_1},G_{p_2}$ and $G_{p_3}$ be independent random subgraphs of $G$ and let us write $G_1=G_{p_1}$, $G_2=G_1\cup G_{p_2}$ and $G_3=G_2\cup G_{p_3}$, where we note that $G_3$ has the same distribution as $G_p$. We define $W_i$ (for $i=1,2,3$) to be the set of vertices in components of order at least $d^{16}$ in $G_i$. We note that $W_1\subseteq W_2\subseteq W_3$.

We begin by showing that \textbf{whp} there are no components of order larger than $O_{\epsilon}(d)$ in $G_p$ which do not meet $W_1$.
\begin{lemma}\label{the gap}
There exists a constant $C_1 \coloneqq C_1(\epsilon, C)$ such that \textbf{whp} every component $M$ of $G_p$ with $|M| \ge C_1 d$ intersects with $W_1$.
\end{lemma}
\begin{proof}
We start by sampling $G_{p_1} = G_1$. Let $V_1\subseteq V(G)\setminus W_1$ be the set of all vertices of $G$ that have at least $\frac{\epsilon^2d}{161C}$ neighbours (in $G$) in $W_1$, and let $V_0=V(G)\setminus(V_1\cup W_1)$.
We first note that, by Lemma \ref{dense2} applied to $G_{p_1}$, \textbf{whp} every connected subset of $G$ with exactly $C_1d$ vertices contains fewer than $\frac{\epsilon d}{19C}$ vertices in $V_0$. We assume henceforth that this property holds.

We then expose the rest of $G_p$ on $V_0 \cup V_1$. There are at most $n$ components of $G_p[V_0 \cup V_1]$. Given a connected component $M$ of $G_p[V_0\cup V_1]$ of order at least $C_1d$, it contains a connected subset $M_0\subseteq M$ with $|M_0|=C_1d$. Since $M_0$ spans a connected subset of $G$, by the above
\[
|M_0 \cap V_1| \geq C_1d - \frac{\epsilon d}{19C} \geq \frac{C_1 d}{2},
\]
if $C_1$ is sufficiently large.

Each vertex in $M_0 \cap V_1$ has at least $\frac{\epsilon^2d}{161C}$ neighbours in $W_1$ and so there is a set $X$ of at least $\frac{\epsilon^2C_1d^2}{330C}$ edges between $M_0$ and $W_1$. We now expose the edges in $X \cap G_{p_2}$. By Lemma \ref{chernoff}, $X \cap G_{p_2} \neq \emptyset$ with probability at least $1-\exp\left(-\frac{\epsilon^3C_1d}{400C}\right)$. Recalling that $d=\Theta_C(\log n)$, we let $\alpha$ be the constant such that $d=\alpha\log n$. 

Then, for $C_1 \geq \frac{800\cdot C\cdot \alpha}{\epsilon^3}$, with probability $1-o\left(\frac{1}{n}\right)$, $M_0$ is adjacent to a vertex in $W_1$ in $G_{p_2} \subseteq G_p$, and therefore $M$ is adjacent to a vertex in $W_1$ in $G_{p_2}$. Hence, by a union bound, \textbf{whp} every component of $G_p[V_0 \cup V_1]$ of order at least $C_1d$ is adjacent in $G_p$ to a vertex in $W_1$, from which the statement follows.
\end{proof}

\begin{lemma}\label{all merge}
\textbf{Whp}, all the components of $G_2[W_2]$ are contained in a single component of $G_p$.
\end{lemma} 
\begin{proof}
We first note that every edge of $G$ belongs to $G_2$ independently with probability $1-(1-p_1)(1-p_2)=\frac{1+\epsilon-o(1)}{d}$. Thus, by Lemma \ref{dense3}, \textbf{whp} every $v\in V(G)$ is at distance at most $2$ from a vertex of $W_2$. We assume henceforth that this property holds.

Suppose the statement of the lemma does not hold, then we can partition the components of $G_2[W_2]$ into two families, denote them by $\mathcal{A}$ and $\mathcal{B}$, such that there are no paths in $G_p$ (and hence in $G_{p_3}$) between $\mathcal{A}$ and $\mathcal{B}$. Denote by $s$ the number of components in $\mathcal{A}\cup \mathcal{B}$, and let $\ell\le \frac{s}{2}$ be the number of components in the smaller family. Observe that both $A=\cup_{C\in \mathcal{A}}V(C)$ and $B=\cup_{C\in \mathcal{B}}V(C)$ have at least $\ell d^{16}$ vertices in them.

Now, since every vertex in $G$ is at distance at most $2$ from either $A$ or $B$, we can partition $V(G)$ into two sets $A'$ and $B'$, such that $A'$ contains $A$ and all $v\in V(G)\setminus B$ at distance at most $2$ from $A$, and similarly $B'$ contains $B$ and all $v\in V(G)\setminus A'$ at distance at most $2$ from $B$.

By Theorem \ref{productiso}, $i(G)\ge \frac{1}{2}\min_{j\in [t]}i\left(G^{(j)}\right)$. Since for all $j\in[t]$, $1<|V\left(G^{(j)}\right)|\le C$ and $G^{(j)}$ is connected, we have that $i\left(G^{(j)}\right)\ge\frac{1}{C}$ and thus $i(G)\ge \frac{1}{2C}$ (note that this is the only place where we use our assumption on the connectedness of the base graphs). It follows that there are at least $\frac{\ell d^{16}}{2C}$ edges between $A'$ and $B'$. Now, since every vertex in $A'$ is at distance at most $2$ from $A$, and similarly every vertex in $B'$ is at distance at most $2$ from $B$, we can extend these edges to a family of $\frac{\ell d^{16}}{2C}$ paths of length at most $5$ between $A$ and $B$. Since $G$ is $d$-regular, every edge participates in at most $5d^4$ paths of length at most $5$.

Thus, we can greedily thin this family to a set of $\frac{\ell d^{16}}{2C\cdot 5\cdot 5d^4}=\frac{\ell d^{12}}{50C}$ edge-disjoint paths of length at most $5$ between $A$ and $B$. The probability none of these paths are in $G_{p_3}$ is thus at most
\begin{align*}
    \left(1-p_3^5\right)^{\frac{\ell d^{12}}{50C}}\le \exp\left(-\frac{\ell d^{2}}{50C}\right).
\end{align*}
On the other hand, there are at most $n$ components in $G_2[W_2]$ and hence the number of ways to partition these components into two families with one containing $\ell$ components is at most $\binom{n}{\ell}$. Thus, by the union bound, the probability that the statement does not hold is at most
\begin{align*}
    \sum_{\ell=1}^{\frac{s}{2}}\binom{n}{\ell}\exp\left(-\frac{\ell d^{2}}{50C}\right)\le \sum_{\ell=1}^{\frac{s}{2}}\left[n\exp\left(-\frac{d^{2}}{50C}\right)\right]^{\ell}=o(1),
\end{align*}
since $d=\Theta_C(\log n)$, completing the proof.
\end{proof}

We are now ready to prove Theorem \ref{supercritical regime}:
\begin{proof}[Proof of Theorem \ref{supercritical regime}]
By Lemma \ref{the gap}, \textbf{whp} there are no components in $G_p$ of order larger than $O_{\epsilon}(d)$ which do not intersect with $W_1$. Thus, since every component of $G_p[W_3]$ has size at least $d^{16}$, \textbf{whp} every component of $G_p[W_3]$ meets a component of $G_1[W_1]$, and so (by inclusion) meets a component of $G_2[W_2]$. Hence, by Lemma \ref{all merge}, \textbf{whp} all these components are contained in a single component of $G_p$, which we denote by $L_1$, with $V(L_1)=W_3$. By Lemma \ref{big comp}, we have that \textbf{whp} $|W_3|=\left(1+o(1)\right)yn$. 

Finally, once again by Lemma \ref{the gap}, \textbf{whp} there are no components larger than $O_{\epsilon}(d)$ which do not intersect with $W_1 \subset W_3$, and thus by the above \textbf{whp} there are no components besides $L_1$ of order larger than $O_{\epsilon}(d)$.
\end{proof}

\section{Subcritical regime}\label{S: subcritical}
The proof of Theorem \ref{subcritical regime} is inspired by Krivelevich and Sudakov's \cite{KS13} simple proof of the emergence of a giant component in $G(d+1,p)$. However, here we analyse the BFS algorithm instead of the Depth First Search algorithm. 

\begin{proof}[Proof of Theorem \ref{subcritical regime}]
Suppose we run the BFS algorithm on $G_p$, as described in Section \ref{BFS description}, and assume to the contrary that $G_p$ contains a component $K$ of order larger than $k=\frac{9\log n}{\epsilon^2}$. 

Let us consider the period of the algorithm from the moment when the first vertex of $K$ enters $Q$ to the moment when the $(k+1)$-st vertex of $K$ enters $Q$. During this period we only query edges adjacent to the first $k$ vertices of $K$, and so, since $G$ has maximum degree $d$, we query at most $kd$ edges. However, by assumption $G_p$ induces a tree on these $k+1$ vertices, and hence in this period we receive $k$ positive answers. In particular, there is some interval $I \subseteq \left[m \right]$ of length $kd$ such that at least $k$ of the $X_i$ with $i \in I$ are equal to $1$.

However, by Lemma \ref{chernoff}, the probability that this occurs for any fixed interval $I$ is at most
\begin{align*}
    \mathbb{P}\left[Bin\left(kd,\frac{1-\epsilon}{d}\right)\ge k\right]\le \exp\left(-\frac{\epsilon^2k}{4}\right)=o\left(\frac{1}{n^2}\right),
\end{align*}
where the last equality holds since $k=\frac{9\log n}{\epsilon^2}$. Since $G$ has maximum degree $d$, we have that $m\le \frac{nd}{2}$. Therefore, there are at most $\frac{nd}{2} \leq n^2$ intervals $I \subseteq \left[\frac{nd}{2} \right]$ of length $kd$, and so by a union bound \textbf{whp} there is no such interval, contradicting our assumption.
\end{proof}

\section{Discussion and possible future research directions} \label{discussion}

We have shown that under the assumption that the base graphs are connected, regular, and of bounded order, the percolated product graph $G_p$ exhibits the Erd\H{o}s-R\'enyi phenomenon. 

Let us begin by discussing some of the key steps in the proof of Theorem \ref{supercritical regime}, and their relation to previous results. Given a small set of vertices, the projection lemma (Lemma \ref{seperation lemma}) allows us to find a disjoint set of substructures of `small codimension' which separates them, in the sense that each vertex is contained in a unique substructure. Since our asymptotics are in terms of the dimension of the graph, the percolated substructures display similar properties to $G_p$, and, crucially, they are independently distributed. In a rough sense, this allows us to conclude that there is not much correlation between the local structure around these vertices. This idea is not entirely new and is implicitly used in the classical proofs of the hypercube \cite{AKS81, BKL92}, and in the recent work of Lichev on product graphs \cite{L22}. However, the generality in which it is stated, as well as some novel applications, allows us to provide a streamlined proof which furthermore, as we will soon discuss, works in the absence of a strong isoperimetric inequality.

The first key improvement of Theorem \ref{supercritical regime}, compared with the results of \cite{L22}, is that it recovers the sharp asymptotic order of the giant component. To that end, in Lemma \ref{medium comp} we estimate the probability that a fixed vertex lies in a fairly large component. In the case of the hypercube, this was done in \cite{BKL92} by utilising Harper's isoperimetric inequality \cite{H64}, together with a coupling of the BFS exploration with a Galton-Watson tree. In the proof of Lemma \ref{medium comp}, similarly to \cite{BKL92}, we also couple the BFS exploration with a Galton-Watson tree. Then, however, instead of using any isoperimetric inequality, we inductively utilise the projection lemma. A similar inductive approach was done in the setting of \textit{site percolation} on the hypercube \cite{BKL94}, however, the use of the projection Lemma here allows for a much more compact and streamlined proof. The second key improvement of Theorem \ref{supercritical regime} is in showing that typically the second largest component of $G_p$ is of order $O_{\epsilon}(d)$ (and in this part lies the key innovation in the proof). Utilising the projection lemma, in Lemmas \ref{dense1} and \ref{dense2} we show that it is very unlikely for a set of order $\Omega_{\epsilon}(d)$ to have relatively few neighbours (in $G$) in large components of $G_p$. Then, instead of employing the standard double sprinkling argument (used in \cite{AKS81, BKL92, L22}), we utilise a multi-round sprinkling argument, exposing first the edges inside the set of small components and only then the edges from this set to the rest of the graph. This allows us to show that not only all the large components typically merge, but that they `swallow' any component whose order is $\Omega_{\epsilon}(d)$.

The isoperimetric inequality in Theorem \ref{productiso} is far from optimal in the case of the hypercube. Indeed, whilst Theorem \ref{productiso} is of the right order of asymptotics for linear-sized sets, which can have constant edge-expansion, a well-known result of Harper \cite{H64} implies that smaller sets in $Q^d$ have almost optimal edge-expansion, expanding by a factor of almost $d$. As mentioned above, given such a strong isoperimetric inequality, it is relatively simple to show the existence of a gap in the sizes of the components in $Q^d_p$ using a first-moment argument, which implies any component of polynomial (in $d$) order must in fact have order $O_{\epsilon}(d)$ (see \cite{BKL92}, and also \cite{K23} for a simple proof for the phase transition in $Q^d_p$). Utilising the product structure of the graph, our proof allows us to use a much weaker isoperimetric inequality. We believe these methods could be useful in other contexts where the analysis is constrained by the lack of a strong enough isoperimetric inequality. Indeed, similar arguments have proven to be useful in the setting of site percolation on the hypercube, where the lack of a strong enough vertex isoperimetric inequality complicates the analysis of the component structure. Recently, the first and fourth author \cite{DK22} used similar ideas to verify a longstanding conjecture of Bollob\'as, Kohayakawa, and \L{}uczak \mbox{\cite[Conjecture 11]{BKL94}} on the size of the second-largest component in the supercritical regime in this model. Furthermore, in a subsequent work, the authors \cite{DEKK24b} show that for $d$-regular graphs, where $d=\omega(1)$, a fairly mild isoperimetric assumption on $G$ suffices for phase transition in $G_p$ for the typical emergence of a giant component.

As mentioned in the introduction, Lichev \cite{L22} showed the emergence of a linear sized component around $p=\frac{1}{d}$, where he did not assume regularity or bounded order, but only some mild isoperimetric assumption on the base graphs. In a companion paper, \cite{DEKK23} we give an example to show that the regularity assumption is necessary to a certain extent --- if we take our base graphs to be stars, and so quite irregular, then the component behaviour can be quite different, with polynomially sized components appearing already in the subcritical regime. We also show that if one further allows some of the base graphs to be of unbounded order, then there are product graphs (even satisfying a relatively strong isoperimetric inequality), in which typically a largest component in the supercritical regime is of sublinear order.  Furthermore, using some of the tools developed in this paper we are also able to significantly relax the isoperimetric requirements on the base graphs from a polynomial dependence in the dimension to a superexponential one. It remains an interesting open question as to whether we can relax the assumption that the base graphs have bounded size to that of just bounded regularity, or even weaker to some assumption of \emph{almost-regularity} of the product graph, under some mild isoperimetric assumptions.
\begin{question}
Let $G$ be a high-dimensional product graph all of whose base graphs are connected and let $d$ be the average degree of $G$. Let $\epsilon > 0$ and let $p=\frac{1+\epsilon}{d}$. What assumptions on degree distributions and the isoperimetric constants of the base graphs are sufficient to guarantee that $G_p$ exhibits the Erd\H{o}s-R\'{e}nyi component phenomenon?
\end{question}
More generally, it would be interesting to investigate further which other classes of graph exhibit the Erd\H{o}s-R\'enyi component phenomenon under percolation and to find the limits of the universality of this phenomenon.

Finally, whilst the Cartesian product is perhaps a natural product in the context of percolation, there are many other types of graphs products, such as strong products and tensor products, and it would be interesting to know if `high-dimensional' graphs with respect to these types of products also exhibit similar behaviour under percolation. As a concrete example, since the $d$-fold tensor product of a single edge is disconnected and the $d$-fold strong product of a single edge is complete, it is perhaps interesting to consider percolation in the tensor and strong products of many small cycles. Let us write $T(k,d)$ for the $d$-fold tensor product of the cycle $C_k$ and similarly $S(k,d)$ for the $d$-fold strong product. Note that $T(k,d)$ is $2^d$-regular and $S(k,d)$ is $2^d + 2d$ regular. 
\begin{question}
Do the percolated subgraphs $T(3,d)_p$ and $S(3,d)_p$ exhibit the Erd\H{o}s-R\'{e}nyi component phenomenon at the critical point $p=2^{-d}$?
\end{question}

\paragraph{Acknowledgement} The authors thank the anonymous referees for their helpful comments that improved the quality of the manuscript. The second and third authors were supported in part by the Austrian Science Fund (FWF) [10.55776/{\text{P36131, W1230, I6502}]. The fourth author was supported in part by USA–Israel BSF grant 2018267. For the purpose of open access, the authors have applied a CC-BY public copyright licence to any Author Accepted Manuscript version arising from this submission.
\bibliographystyle{abbrv}
\bibliography{perc}

\begin{thebibliography}{10}

\bibitem{AKS81}
M.~Ajtai, J.~Koml\'{o}s, and E.~Szemer\'{e}di.
\newblock Largest random component of a {$k$}-cube.
\newblock {\em Combinatorica}, 2(1):1--7, 1982.

\bibitem{AS16}
N.~{Alon} and J.~H. {Spencer}.
\newblock {\em {The probabilistic method}}.
\newblock Hoboken, NJ: John Wiley \& Sons, fourth edition, 2016.

\bibitem{BFM98}
A.~{Beveridge}, A.~{Frieze}, and C.~{McDiarmid}.
\newblock {Random minimum length spanning trees in regular graphs}.
\newblock {\em {Combinatorica}}, 18(3):311--333, 1998.

\bibitem{B01}
B.~Bollob\'{a}s.
\newblock {\em Random graphs}.
\newblock Cambridge Studies in Advanced Mathematics. Cambridge University
  Press, Cambridge, second edition, 2001.

\bibitem{BKL92}
B.~Bollob\'{a}s, Y.~Kohayakawa, and T.~{\L}uczak.
\newblock The evolution of random subgraphs of the cube.
\newblock {\em Random Structures Algorithms}, 3(1):55--90, 1992.

\bibitem{BKL94}
B.~Bollob\'{a}s, Y.~Kohayakawa, and T.~{\L}uczak.
\newblock On the evolution of random {B}oolean functions.
\newblock In {\em Extremal problems for finite sets ({V}isegr\'{a}d, 1991)},
  volume~3 of {\em Bolyai Soc. Math. Stud.}, pages 137--156. J\'{a}nos Bolyai
  Math. Soc., Budapest, 1994.

\bibitem{BR06}
B.~{Bollob\'as} and O.~{Riordan}.
\newblock {\em {Percolation.}}
\newblock Cambridge University Press, Cambridge, 2006.

\bibitem{BH57}
S.~B. {Broadbent} and J.~M. {Hammersley}.
\newblock {Percolation processes. I: Crystals and mazes}.
\newblock {\em {Proc. Camb. Philos. Soc.}}, 53:629--641, 1957.

\bibitem{CT98}
F.~R.~K. Chung and P.~Tetali.
\newblock Isoperimetric inequalities for {C}artesian products of graphs.
\newblock {\em Combin. Probab. Comput.}, 7(2):141--148, 1998.

\bibitem{DEKK24b}
S.~Diskin, J.~Erde, M.~Kang, and M.~Krivelevich.
\newblock Percolation through isoperimetry.
\newblock {\em arxiv:2308.10267}, 2023.

\bibitem{DEKK23}
S.~Diskin, J.~Erde, M.~Kang, and M.~Krivelevich.
\newblock Percolation on irregular high-dimensional product graphs.
\newblock {\em Combin. Probab. Comput.}, 33(3):377--403, 2024.

\bibitem{DK22}
S.~Diskin and M.~Krivelevich.
\newblock Supercritical site percolation on the hypercube: small components are
  small.
\newblock {\em Combin. Probab. Comput.}, 32(3):422--427, 2023.

\bibitem{D19}
R.~Durrett.
\newblock {\em Probability: theory and examples}.
\newblock Cambridge University Press, Cambridge, 2019.

\bibitem{ER60}
P.~Erd\H{o}s and A.~R\'{e}nyi.
\newblock On the evolution of random graphs.
\newblock {\em Magyar Tud. Akad. Mat. Kutat\'{o} Int. K\"{o}zl.}, 5:17--61,
  1960.

\bibitem{ES79}
P.~Erd\H{o}s and J.~Spencer.
\newblock Evolution of the {$n$}-cube.
\newblock {\em Comput. Math. Appl.}, 5(1):33--39, 1979.

\bibitem{FK16}
A.~Frieze and M.~Karo\'{n}ski.
\newblock {\em Introduction to random graphs}.
\newblock Cambridge University Press, Cambridge, 2016.

\bibitem{FKM04}
A.~Frieze, M.~Krivelevich, and R.~Martin.
\newblock The emergence of a giant component in random subgraphs of
  pseudo-random graphs.
\newblock {\em Random Structures Algorithms}, 24(1):42--50, 2004.

\bibitem{G99}
G.~{Grimmett}.
\newblock {\em {Percolation.}}
\newblock Springer, Berlin, 1999.

\bibitem{H64}
L.~H. Harper.
\newblock Optimal assigments of numbers to vertices.
\newblock {\em SIAM J. Appl. Math.}, 12:131--135, 1964.

\bibitem{H63}
T.~E. Harris.
\newblock {\em The theory of branching processes}, volume 119 of {\em
  Grundlehren Math. Wiss.}
\newblock Springer, Cham, 1963.

\bibitem{JLr00}
S.~Janson, T.~{\L}uczak, and A.~Ruci\'{n}ski.
\newblock {\em Random graphs}.
\newblock Wiley-Interscience Series in Discrete Mathematics and Optimization.
  Wiley-Interscience, 2000.

\bibitem{K82}
H.~{Kesten}.
\newblock {\em {Percolation theory for mathematicians}}.
\newblock Birkh\"auser, Boston, MA, 1982.

\bibitem{K23}
M.~Krivelevich.
\newblock Component sizes in the supercritical percolation on the binary cube.
\newblock {\em arXiv preprint arXiv:2311.07210}, 2023.

\bibitem{KS13}
M.~Krivelevich and B.~Sudakov.
\newblock The phase transition in random graphs: {A} simple proof.
\newblock {\em Random Structures Algorithms}, 43(2):131--138, 2013.

\bibitem{L22}
L.~Lichev.
\newblock The giant component after percolation of product graphs.
\newblock {\em J. Graph Theory}, 99(4):651--670, 2022.

\bibitem{MR95}
M.~Molloy and B.~Reed.
\newblock A critical point for random graphs with a given degree sequence.
\newblock {\em Random Structures Algorithms}, 6(2-3):161--179, 1995.

\bibitem{NP10}
A.~Nachmias and Y.~Peres.
\newblock Critical percolation on random regular graphs.
\newblock {\em Random Structures Algorithms}, 36(2):111--148, 2010.

\end{thebibliography}
\end{document}